\documentclass[11pt,reqno]{amsart}
\usepackage{amsmath}
\usepackage{amsthm}
\usepackage{amscd}
\usepackage{amssymb}
\usepackage{graphicx}
\usepackage{epsfig}
\theoremstyle{plain}
\newtheorem{theorem}{Theorem}[section]
\newtheorem{proposition}[theorem]{Proposition}
\newtheorem{lemma}[theorem]{Lemma}

\newtheorem{definition}[theorem]{Definition}

\catcode`\@=11
\long\def\@savemarbox#1#2{\global\setbox#1\vtop{\hsize\marginparwidth 
  \@parboxrestore\tiny\raggedright #2}}
\newcommand\fakelabel[2]{\@bsphack\if@filesw {\let\thepage\relax
   \newcommand\protect{\noexpand\noexpand\noexpand}%
\xdef\@gtempa{\write\@auxout{\string
      \newlabel{#1}{{#2}{\thepage}}}}}\@gtempa
   \if@nobreak \ifvmode\nobreak\fi\fi\fi\@esphack}
\def\SL@margintext#1{{\showlabelsetlabel{\tiny\{\SL@prlabelname{#1}\}}}}
\catcode`\@=12

\def\int{\text{interior}}
\def\hyp {\hbox {\rm {H \kern -2.8ex I}\kern 1.25ex}}
\def\reals {\hbox {\rm {R \kern -2.8ex I}\kern 1.15ex}}
\def\integers {\hbox {\rm { Z \kern -2.8ex Z}\kern 1.15ex}}
\def\naturals {\hbox {\rm {N \kern -2.8ex I}\kern 1.20ex}}
\def\rationals {\hbox {\rm { Q \kern -2.2ex l}\kern 1.15ex}}
\def\hyp {\hbox {\rm {H \kern -2.7ex I}\kern 1.25ex}}
\def\bar{\overline}

\def\hat{\widehat}

\def\CC{\mathcal{C}}
\def\AC{\mathcal{AC}}
\def\union{\cup}
\def\intersect{\cap}
\def\boundary{\partial}
\def\DD{\mathcal{D}}
\def\ssm{\smallsetminus}
\def\LL{\mathcal{L}}
\newcommand{\diam}{\operatorname{diam}}

\def\strutdepth{\dp\strutbox}
\def \ss{\strut\vadjust{\kern-\strutdepth \sss}}
\def \sss{\vtop to \strutdepth{
\baselineskip\strutdepth\vss\llap{$\diamondsuit\;\;$}\null}}

\def\strutdepth{\dp\strutbox}
\def \sst{\strut\vadjust{\kern-\strutdepth \ssss}}
\def \ssss{\vtop to \strutdepth{
\baselineskip\strutdepth\vss\llap{$\spadesuit\;\;$}\null}}

\def\strutdepth{\dp\strutbox}
\def \ssh{\strut\vadjust{\kern-\strutdepth \sssh}}
\def \sssh{\vtop to \strutdepth{
\baselineskip\strutdepth\vss\llap{$\heartsuit\;\;$}\null}}

\begin{document}

\title{Heegaard splittings with large subsurface distances}

\author{Jesse Johnson}

\address{\hskip-\parindent
        Department of Mathematics \\
        Oklahoma State University\\
        Stillwater, OK 74078\\
        USA}
\email{jjohnson@math.okstate.edu}

\author{Yair Minsky}
\address{\hskip-\parindent
        Department of Mathematics \\
        Yale University \\
        PO Box 208283 \\
        New Haven, CT 06520 \\
        USA}
\email{yair.minsky@yale.edu}

\author{Yoav Moriah}
\address{\hskip-\parindent
        Department of Mathematics \\
        Technion \\
        Israel}
\email{ymoriah@techunix.technion.ac.il}


\thanks{The first author was partially supported by NSF MSPRF grant   0602368.
The second author was partially supported by NSF grant 0504019.
The third author would like to thank Yale University for its hospitality during
 a sabbatical in which the research was done.}

\begin{abstract} We show that subsurfaces of a Heegaard 
  surface for which the relative Hempel distance of the splitting
 is sufficiently high  have to appear in
  any Heegaard surface  of genus bounded by half that distance.  
\end{abstract}

\maketitle

\section{Introduction}

 It was shown by Scharlemann and Tomova  ~\cite{ST}  that if a  $3$-manifold has a Heegaard  surface 
 $\Sigma$ so  that the Hempel distance $d(\Sigma)$ of the splitting (see \cite{hempel:complex}) is greater 
 than  twice  its genus $g(\Sigma)$,  then any other Heegaard surface $\Lambda$ of  genus  
 $ g(\Lambda)  < d(\Sigma) / 2$  is a stabilization of $\Sigma$, i.e., the result of attaching trivial  handles 
 to $\Sigma$.  Hartshorn \cite{hartshorn} had proved a similar result in which the surface $\Lambda$ is
 incompressible.

In this paper, we generalize this theorem to the case where theHeegaard splitting $\Sigma$ is not necessarily of high distance but has a proper essential subsurface $F \subset \Sigma$ so that the ``subsurface distance'' measured in $F$ is large:

\begin{theorem} \label{mainthm}
Let $\Sigma$ be a Heegaard  surface in a $3$-manifold $M$ of genus $g(\Sigma) \ge 2$, 
and let $F \subset \Sigma$ be a compact essential subsurface. Let  $\Lambda$ be
another  Heegaard surface for $M$ of genus $g(\Lambda)$.  If
$$ d_F(\Sigma) > 2g(\Lambda) + c(F), $$
then, up to ambient isotopy,  the intersection $\Lambda \cap \Sigma$ contains $F$. 

Here $c(F) = 0$ unless $F$ is an annulus, 4-holed sphere, or  1 or 2-holed torus, in which case $c(F) = 2$. 

\end{theorem}

Here, $d_F(\Sigma)$ is the distance between the subsurface projections of the disk sets of the handlebodies on the two sides of $\Sigma$ to the ``arc and curve complex'' of $F$.   For precise definitions see Section
 \ref{HS}. This result can be paraphrased as follows:  If the two disk sets of a Heegaard splitting intersect on a subsurface of the Heegaard surface in a relatively ``complicated'' way, then any other Heegaard surface whose genus is not too large must contain that subsurface.

One can view this in the context of a number of results in 3-manifold topology and geometry where ``local'' complication, as measured in a subsurface, yields some persistent topological or geometric feature. An example with a combinatorial flavor is in Masur-Minsky \cite{masur-minsky:complex1,masur-minsky:complex2}, where one studies certain quasi-geodesic paths in the space of markings on a surface. If the endpoints of such a path have projections to the arc and curve complex of a subsurface $F$ which are sufficiently far apart, then the path is forced to have a long subinterval in which $\boundary F$ is a part of the markings. 

In \cite{minsky:ELCI, brock-canary-minsky:ELCII}, Brock-Canary-Minsky  consider the geometry of hyperbolic structures on $\Sigma\times{\mathbb R}$, as controlled by their ``ending invariants,'' which can be thought of as markings on $\Sigma$. A subsurface $F$ for which the projections of the ending invariants are far apart yields a ``wide'' region in the manifold isotopic to $F\times [0,1]$, where the length of $\boundary F$ is very
short. 

A similar, but incompletely understood, set of phenomena can occur for Heegaard splittings of hyperbolic 
3-manifolds, where the disk sets play the role of the ending invariants. This is the subject of work  in
progress by Brock, Namazi, Souto and others. 

In the purely topological setting, Lackenby \cite{lackenby:amalgamated} studies Heegaard  splittings of manifolds obtained by sufficiently complicated gluings along incompressible, acylindrical boundaries of genus at least 2. Bachman-Schleimer \cite{bachman-schleimer:bundles} obtain lower bounds on Heegaard genus in a surface bundle whose monodromy map is sufficiently complicated. Bachman-Schleimer-Sedgwick
\cite{bachman-schleimer-sedgwick} study Heegaard genus for manifolds glued along tori by complicated maps.

The methods of our proof are extensions, via subsurface projections,  of the work in Johnson 
\cite{johnson:flipping}, which itself builds on methods of Rubinstein-Scharlemann \cite{rubinstein-scharlemann:graphic}. To a Heegaard splitting we associate a ``sweepout'' by parallel surfaces of the manifold minus a pair of spines, and given two surfaces $\Lambda$ and $\Sigma$ and associated sweepouts,  we examine the way in which they interact.  In particular under some natural genericity conditions we can assume
that one of two situations occur: 

\begin{enumerate}
\item  $\Lambda$ {\em $F$-spans} $\Sigma$, or

\item $\Lambda$ {\em $F$-splits} $\Sigma$.
\end{enumerate}

\vskip7pt

For precise definitions see Section \ref{above below}. The case of $F$-spanning implies that, up to isotopy, there is a moment in the sweepout corresponding to  $\Sigma$ when the subsurface parallel to $F$ lies in the upper half of $M\ssm \Lambda$, and a moment when it lies in the lower half. It then follows that $\Lambda$ separates the product region between these two copies of $F$, and it follows fairly easily that it be isotoped to contain $F$. 

In the $F$-splitting case we are able to show that, for each moment in the sweepout corresponding to 
$\Sigma$, the level surface intersects $\Lambda$ in curves that have essential intersection with $F$. By studying the way in which these intersections change during the course of the sweepout, we are able to
use the topological complexity of $\Lambda$ to control the subsurface distance of $F$.

Combining these two results and imposing the condition that $d_F(\Sigma)$ is greater than a suitable function of $g(\Lambda)$ forces the first, i.e. $F$-spanning, case to occur.

The discussion is complicated by some special cases, where $F$ has particularly low complexity, in which the dichotomy between $F$-spanning and $F$-splitting doesn't quite hold. In those cases we obtain slightly different bounds. 

\medskip

\subsection*{Outline}
In Section \ref{HS} we recall the definitions of Heegaard splittings, curve complexes and subsurface projections.  In Section \ref{above below} we discuss sweepouts, pairs of sweepouts and the Rubinstein-Scharlemann graphic, and use this to define $F$-splitting and  $F$-spanning and variations. In Lemma 
\ref{dichotomy} we show that $F$-spanning and $F$-splitting are generically complementary conditions, and also work out the situation for the low-complexity cases. 

In Section \ref{spanning case} we consider the $F$-spanning case, and show that it leads to the conclusion that $\Lambda$ can be isotoped to contain $F$. 

In Section \ref{saddle transitions} we give a more careful analysis of the local structure of non-transverse intersections of surfaces in a pair of sweepouts. In particular we use this to quantify the way in which intersection loops, and their projections to a subsurface $F$, can change as one moves the sweepout surfaces. 

In Section \ref{boundsect} we show that the $F$-splitting condition leads to an upper bound on $d_F(\Sigma)$. The proof makes considerable use of the structure developed in Section \ref{saddle transitions}, with a fair amount of attention being necessary to handle the exceptional low-complexity cases. 

Finally in Section \ref{proof main} we put these results together to
obtain the proof of the main theorem.

\subsection*{Acknowledgements}
The authors are grateful to Saul Schleimer for pointing out an error
in the original proof of Lemma \ref{disks everywhere}.

\section{Heegaard splittings and curve complexes}\label{HS}

\subsection*{Heegaard splittings}
A handlebody $H$ is a $3$-manifold homeomorphic to a regular neighborhood of a finite, connected 
graph in $S^3$. The graph will be called a {\it spine} of the handlebody. A compression body 
$W$ is  a $3$-manifold obtained from a connected surface $S$ crossed with an interval $[0, 1]$ by 
attaching $2$-handles  to $S \times \{0\}$ and capping off $S^2$ boundary components, if they 
occur, by $3$-balls. The connected surface $S \times \{1\}$ is denoted by $\partial _{+}W$ and  
$\partial W \smallsetminus \partial_{+}W$ is denoted by  $\partial_{-}W$.  A {\it spine} of a compression 
body $W$ is the union of $\partial_{-}W$ with a (not necessarily
connected) graph $\Gamma$  so that  $W$ is homeomorphic to a closed regular neighborhood of $\partial_{-}W \cup \Gamma$.

A  Heegaard splitting of genus $g \geq2 $ for a $3$-manifold $M$ is a triple $(\Sigma, H^-, H^+)$, 
where  $H^+, H^-$ are genus $g$ handlebodies (compression bodies if $\partial M \neq \emptyset$) 
and  $\Sigma = \partial H^+ =  \partial H^- = H^+ \cap H^-$ is the Heegaard surface so that 
$M = H^+ \cup_{\Sigma} H^-$.

\subsection*{Curve complexes and Hempel distance}
For any surface $\Sigma$ there is an associated  simplicial complex called the {\it curve complex} 
and denoted by $\mathcal{C}(\Sigma)$. An $i$-simplex in $\mathcal{C}(\Sigma)$ is a collection 
$([\gamma_0], \dots, [\gamma_i])$ of isotopy classes of mutually disjoint essential simple closed curves 
in $\Sigma$.  On the $1$-skeleton $\mathcal{C}^1(\Sigma)$  of the curve complex $\mathcal{C}(\Sigma)$  there is a natural path metric $d$ defined by assigning length 1 to every edge. The subcollection 
$\mathcal{D}(H^\pm)$, of isotopy classes of curves in $\Sigma$, that bound disks in $H^\pm$ (also 
called {\em   meridians}) is called the {\it  handlebody set} associated with $H^\pm$ respectively.

Given a Heegaard splitting $(\Sigma, H^-,H^+)$ we define the {\it  Hempel distance} $d(\Sigma)$:

$$d(\Sigma) = d_{\CC^1(\Sigma)} ( \DD(H^+),\DD(H^-))$$
(where distance between sets is always minimal distance)

\medskip

When a surface  $F$ has boundary we can define the {\it arc and curve complex}  $\mathcal{AC}(F)$, by considering isotopy classes of essential (non-peripheral) simple closed curves and properly embedded arcs. 
If $F$ is an annulus, there are no essential closed curves, and  the isotopy classes of essential arcs should be
taken rel endpoints. As before an $n$-simplex is a collection of $n + 1$ isotopy classes with   disjoint representative loops / arcs. We denote by $d_F$ the path metric on the 1-skeleton of $\mathcal{AC}^1(F)$ that assigns length 1 to every edge. 

If $F$ is a connected  proper essential subsurface in $\Sigma$, there is a map 
$$ \pi_F:\CC^0(\Sigma) \to \AC^0(F) \union\{\emptyset\} $$
defined as follows (see Masur-Minsky \cite{masur-minsky:complex2} and Ivanov \cite{ivanov:rank,ivanov:subgroupsbook}):  First assume $F$ is not an annulus.  Given a simple closed curve 
$\gamma$ in $\Sigma$, isotope it to intersect $\boundary F$ minimally. If the intersection is empty let 
$\pi_F(\gamma) = \emptyset$. Otherwise consider the isotopy classes of components of  $F\intersect\gamma$ 
as a simplex in $\AC(F)$, and select (arbitrarily) one vertex to be $\pi_F(\gamma)$. If $F$ is an annulus,
 let $\hat F \to \Sigma$ be the annular cover associated to $F$, compactified naturally using the boundary at
infinity of ${\mathbb H}^2$,  lift $\gamma$ to a collection of disjoint properly embedded arcs in $\hat F$, and arbitrarily select an essential one (which must exist if $\gamma$ crosses $F$ essentially) to be 
$\pi_F(\gamma)$. If $\gamma$ does not cross $F$ essentially we again let $\pi_F(\gamma) = \emptyset$. 
By abuse of notation we identify $\AC(\hat F)$ with $\AC(F)$.

Now if  $(\Sigma, H^-, H^+)$ is a Heegaard splitting of a 3-manifold $M$ and $F \subset \Sigma$ is a connected proper essential  subsurface,  let $\mathcal{D}_F(H^-)  = \pi_F(\DD(H^-))$  and 
$\mathcal{D}_F(H^+)  = \pi_F(\DD(H^+))$ -- the projections to $F$ of all loops that bound essential disks in $H^-$ and $H^+$, respectively. The \textit{$F$-distance} of $\Sigma$, which we  will write
$d_F(\Sigma)$, is the distance between these two sets, 
$$ d_F(\Sigma) = d_{\AC^1(F)} ( \DD_F(H^+), \DD_F(H^-)). $$

We will have use for  the following fact, which is a variation on a result of
Masur-Schleimer:

\begin{lemma}\label{disks everywhere}
Let $\Sigma$ be the boundary of a handlebody $H$ and let $F\subset \Sigma$ be an essential connected subsurface of $\Sigma$ which is not  a 3-holed sphere. If $\Sigma \smallsetminus F$ is compressible in $H$
then $\pi_F(\DD(H))$ comes within distance 2 of every vertex of $\AC(F)$, provided $F$ is not a 4-holed sphere. If it is a 4-holed sphere the distance bound is 3. 

\end{lemma}

\begin{proof}
We give the proof in the non-annular case first.  We claim that $\Sigma\ssm F$ must contain a meridian 
$\delta\in \DD(H)$ such that $ F$ can be connected by a path to either side of $\delta$: Compressibility of 
$\Sigma \ssm F$ yields some meridian, which has the desired property if it is nonseparating in $\Sigma$. 
If it is separating,  then on the side complementary to $F$ we can find a nonseparating meridian. 

Let $\gamma_1$ and $\gamma_2$ be components of $\boundary F$ which can be connected by paths to 
the two sides of $\delta$ (possibly $\gamma_1=\gamma_2$). If $\alpha$ is any essential embedded arc in $F$ with endpoints on $\gamma_1$ and $\gamma_2$ then $\alpha$ can be extended to an embedded arc meeting $\delta$ on opposite sides. A band sum between two copies of $\delta$ along this arc yields a new meridian $\delta'$.

If $\delta$ is not isotopic to either $\gamma_1$ or $\gamma_2$, then its essential intersection with $F$ is two copies of the arc $\alpha$, so that $\pi_F(\delta') = \alpha$.  If it is isotopic to $\gamma_1$ or $\gamma_2$ (or both), then $\pi_F(\delta')$ is an arc or closed curve contained in a regular neighborhood of 
$\alpha\union\gamma_1\union \gamma_2$. 

We now need to show that for every vertex $\beta\in\AC(F)$, there is a path of length at most 2 from $\beta$ to $\pi_F(\delta')$, for some choice of $\alpha$. 

Suppose first that $\gamma_1=\gamma_2$. In this case $\delta$ cannot be isotopic to $\gamma_1$, because $\gamma_1$ cannot be connected to both sides of $\delta$ without going through another boundary
component of $F$. Hence in this case $\pi_F(\delta') = \alpha$. Since $F$ is not an annulus, cutting along 
$\beta$ gives a surface no component of which is a disk; hence an essential arc $\alpha$ disjoint from $\beta$ and with endpoints on $\gamma_1$ exists. This gives $d_{\AC(F)}(\beta, \pi_F(\DD(H))) \le 1$. 

Suppose now $\gamma_1\ne \gamma_2$. If $\beta$ is disjoint from $\gamma_1$ and $\gamma_2$ and doesn't separate them,  then we can choose $\alpha$ connecting $\gamma_1$ and $\gamma_2$ and disjoint from $\beta$, so that $d_{\AC(F)}(\beta, \pi_F(\delta')) \le 1$. 

If $\beta$ does separate $\gamma_1$ from $\gamma_2$, let $X_1$ and $X_2$ be the components of 
$F\ssm \beta$. If one of them is non-planar, it contains a nonseparating closed curve $\beta'$, so that 
$d_{\AC(F)}(\beta',\pi_F(\delta')) \le 1$ and hence $d_{\AC(F)}(\beta,\pi_F(\delta')) \le 2$. 

If both $X_i$ are planar, and $F$ is not a 4-holed sphere, then at least one, say $X_1$, has two boundary components other than $\beta$ and $\gamma_1$ or $\gamma_2$. An arc $\beta'$ connecting those boundary
components is disjoint from $\gamma_1$ and $\gamma_2$ and does not separate them, so again we are done using the previous cases. 

Now suppose $\beta$ is an arc with endpoint on $\gamma_1$ or  $\gamma_2$, or both. If it has endpoints on both, then choose $\alpha = \beta$, and note that $\beta$ is disjoint from $\pi_F(\delta')$ in all cases. 
If $\beta$ meets just $\gamma_1$, say, let $\alpha$ be an arc connecting $\gamma_1$ and $\gamma_2$ and disjoint from $\beta$, and consider a regular neighborhood $N$ of
 $\gamma_1\union \gamma_2\union \alpha\union\beta$. Then $\boundary N\ssm\boundary F$ is either an arc
or a pair of closed curves which is disjoint from $\beta, \gamma_1, \gamma_2$ and $\alpha$. The arc, and at least one of the closed curves, is essential as long as $F$ is not a 3-holed sphere. Let $\beta'$ be this curve or arc, and again we see that the distance from $\beta$ to $\pi_F(\delta')$, via $\beta'$, is 2. 

We have yet to handle the case that $F$ is a 4-holed sphere, and $\beta$ separates $\gamma_1$ from 
$\gamma_2$. In this case, it is possible that $\pi_F(\delta')$ is a closed curve separating
$\gamma_1\union\gamma_2$ from the other two boundary components, and in this case one can check that $d_{\AC(F)}(\beta, \pi_F(\delta')) = 3$. 

Finally we consider the case that $F$ is an annulus.  As before there is a meridian $\delta$ that does not separate the boundaries of $F$ in $\Sigma\ssm F$. We may assume that $\delta$ is not isotopic to
$\boundary F$, because if it were then its complement would contain other meridians disjoint from $F$. 
Let $\alpha'$ be an embedded arc connecting the two sides of $\delta$ and passing essentially through $F$, and let $\delta'$ be the result of the band sum, as before. 

Let $\hat F\to \Sigma$ be the compactified annular cover associated to $F$. There is a lift $\hat\alpha$ of 
$\alpha'$ to $\hat F$ connecting lifts $\hat\delta_1$ and $\hat\delta_2$ of $\delta$ which meet opposite
sides of $\boundary \hat F$. There is then a lift $\hat\delta'$ of $\delta'$ disjoint from 
$\hat \delta_1\union\hat\alpha\union\hat\delta_2$, and connecting opposite sides of $\hat F$. This is a representative of $\pi_F(\delta')$.

A Dehn twist of $\alpha'$ around $F$ has the following effect on $\hat\delta'$: it performs a single Dehn twist about the core of $\hat F$, as well as a homotopy of the endpoints of $\hat\delta'$ which stays outside of the disks $D_i$ cobounded by $\hat\delta_i$ and $\boundary\hat F$ ($i=1,2$). 

Now let $\beta$ be an arc representing a vertex of $\AC(\hat F)$. There exists a disjoint arc $\beta'$ whose endpoints lie in $D_1$ and $D_2$. It follows that after $n$ Dehn twists on $\alpha'$ we obtain an arc 
$\alpha_n$ whose associated $\delta'_n$ has lift $\hat\delta'_n$ which is
disjoint from $\beta'$. We conclude that $d_{\AC(F)}(\beta,\pi_F(\delta'_n) \le 2$. 

\end{proof}

\section{Sweepout pairs} \label{above below}

In this section we discuss {\em sweepouts} of a 3-manifold representing a Heegaard splitting, and consider the interaction of pairs of sweepouts using the {\em Rubinstein-Scharlemann graphic} and generalizations of the
notions of {\em mostly above} and {\em mostly below} from Johnson \cite{johnson:flipping}. We will formulate
relative versions of these which allow us to consider subsurfaces, and from this develop a relative version of the {\em spanning} and {\em splitting} relations from \cite{johnson:flipping}.

\subsection*{Sweepouts}
Given a $3$-manifold $M$,a {\it sweepout} of $M$ is a smooth function $f : M \to  [ -1, 1]$ so that each
$t \in (-1, 1)$ is a regular value, and the level set $f^{-1}(t)$ is a Heegaard surface. Furthermore each of the  sets  $\Gamma^+ = f^{-1}(1)$ and set $\Gamma^- = f^{-1}(-1)$ are spines of the respective compression 
bodies. When this happens we say the sweepout {\em represents} the Heegaard splitting associated to each level surface. It is clear (with a bit of attention to smoothness at the spines) that every Heegaard
splitting can be represented this way.

\medskip

Two sweepouts $f$ and $h$  of $M$ determine a smooth function  $f \times h :  M \to [-1, 1] \times [-1, 1]$.  The differential $D(f\times h)$ has rank 2 (or $\dim Ker(D(f\times h))$ = 1) wherever the level sets of $f$ and $h$ are transverse. Thus we define the {\em   discriminant set} $\Delta$ to be the set of points of $M$ for which 
 $\dim Ker(D(f \times h)) > 1$. The discriminant, and its image in $[-1,1]\times[-1,1]$, therefore encode the configuration of tangencies of the level sets of $f$ and $h$. 

A smooth function $\varphi : M \to N$ between smooth manifolds $M, N$ is {\it stable} if there is  
a neighborhood $U$ of $\varphi$ in $C^{\infty}(M, N)$ such that any map $\eta \in U $ is {\em isotopic} to 
$ \varphi$ through a family of maps in $U$. (We say that $\eta$ and $\varphi$ are isotopic when there are diffeomorphism $\beta:M\to M$ and $\alpha:N\to N$, isotopic to the identity, such that 
$\alpha\circ\eta\circ\beta = \varphi$.) Here the topology on $C^\infty(M,N)$ is the Whitney topology, also known as the strong topology on $C^\infty$. This topology differs from the weak (compact-open) topology on 
$C^\infty$ for non-compact domains, which is relevant here since we consider stability on the complement of the spines. 

Kobayashi-Saeki~\cite{KS} show that, after isotopies of $f$ and $h$, one can arrange that $f \times h$ is stable on the complements of the four spines.  When that holds, the kernel of $D(f\times h)$ (off the spines) is always of dimension at most 2  and it follows from Mather~\cite{Ma} that $\Delta$  is a smooth manifold of dimension one.

The image $(f \times h)(\Delta)$ is a graph in  $[ -1, 1] \times [-1,   1]$  with smooth edges, called the 
{\it graphic}, or the {\em Rubinstein-Scharlemann graphic} (See  Rubinstein-Scharlemann~\cite{RS}). 

We call $f\times h$ {\em generic} if it is stable away from the spines and   each arc $[-1,1] \times \{s\}$
in the square intersects at most one vertex of the graphic.  The following lemma of Kobayashi-Saeki \cite{KS} justifies this term:

\begin{lemma}\label{generic is generic}
Any pair of sweepouts can be isotoped to be generic. 

\end{lemma}

Suppose therefore that $f\times h$ is generic.  Points in the square are denoted by $(t,s)$, and we define the surfaces $\Lambda_s = h^{-1}(s)$ and $\Sigma_t = f^{-1}(t)$.  If the vertical line $\{t\}\times[-1,1]$ meets no vertices of the graphic, then $h|_{\Sigma_t}$ is Morse, and its critical points are $\Sigma_t \intersect \Delta$. In particular the fact that a Morse function has at most one singularity at any level corresponds to the fact that the
map from $\Delta$ to the graphic is one-to-one over the smooth points of the graphic.  If $v=(t_0,s_0)$ is a vertex, we see certain transitions as $t$ passes through $t_0$.

\begin{enumerate}

\item {\em Cancelling pair:} If $v$ has valence 2, then 
the edges of the graphic incident to $v$ are either both contained to the left of $\{t_0\}\times(0,1)$, or both to the right of it. As $t$ passes through $t_0$, a pair of singularities of $h|_{\Sigma_t}$, one saddle and one min or max (which we call central), is created or annihilated.

\item {\em Simultaneous singularities:} If $v$ has valence 4,  it is the intersection of the images of two disjoint arcs of $\Delta$. Hence there are two singularities whose relative heights are exchanged as $t$ passes through $t_0$. These can be either saddle or central singularities. Note that the singularities cannot coalesce, for example in a monkey saddle, as they pass each other, since that would produce a vertex in the discriminant set $\Delta$, contradicting the fact that it is a smooth 1-manifold. 

\end{enumerate}

Vertices on the boundary of the square can  have valence either one or two; we will however not need to consider these.

\subsection*{Above and below}
Let  $f$ and $h$ be sweepouts for a manifold $M$, representing the Heegaard splittings   
$(\Sigma, H^-, H^+)$ and   $(\Lambda,V^-,V^+)$, respectively.  For each $s \in (-1,1)$, define  
 $V^-_s = h^{-1}([-1,s])$ and $V^+_s = h^{-1}([s,1])$.  Note that $\Lambda_s = \partial V^-_s = \partial V^+_s$
Similarly, for 
 $t \in (-1,1)$,  $H^-_t = h^{-1}([-1,t])$, $H^+_t = h^{-1}([t,1])$ and $\Sigma_t = \partial H^-_t = \partial H^+_t$.

Throughout the rest of the paper , let $F \subset \Sigma$ denote a compact, connected, essential subsurface, where ``essential'' means that no boundary component of $F$ bounds a disk. We exclude 3-holed spheres. Let $F_t$ be the image of $F$ under the identification of $\Sigma$ with $\Sigma_t$ determined up to isotopy by the sweepout.

\begin{definition}\label{relative mostly}\rm
We will say that $\Sigma_t$ is \textit{mostly above $\Lambda_s$ with respect   to $F$}, denoted 
$$\Sigma_t \succ_F \Lambda_s,$$ if $\Sigma_t \cap V^-_s$ is contained in a subsurface of $\Sigma_t$ that is isotopic into the complement of $F_t$ (or is just contained in a disk, when $F=\Sigma$).  Similarly, $\Sigma_t$ is \textit{mostly below $\Lambda_s$ with respect to $F$}, or
$$\Sigma_t \prec_F \Lambda_s,$$
if $\Sigma_t \cap V^+_s$ is contained in a subsurface that is isotopic into the complement of $F_t$ (or contained in a disk). 

\end{definition}

The case that $F=\Sigma$ corresponds to the notion of {\em mostly  above} and {\em mostly below} from 
\cite{johnson:flipping}. We will mostly be concerned with the case that $F$ is a proper subsurface. 

 Define  $R_a^F $ (respectively $R_b^F$) in $ (-1,1) \times (-1,1)$ to be the set of all  values $(t,s)$ such that 
 $\Sigma_t  \succ_F \Lambda_s $ (respectively $\Sigma_t  \prec_F  \Lambda_s $).   When $F = \Sigma$
these are the sets $R_a$ and $ R_b$ of \cite{johnson:flipping}.

Figure \ref{spanningfig} illustrates some typical configurations of $R_a^F$ and $R_b^F$. Their basic properties are described in the following lemma. 

\begin{figure}[htp]
{\epsfxsize = 3.5 in \centerline{\epsfbox{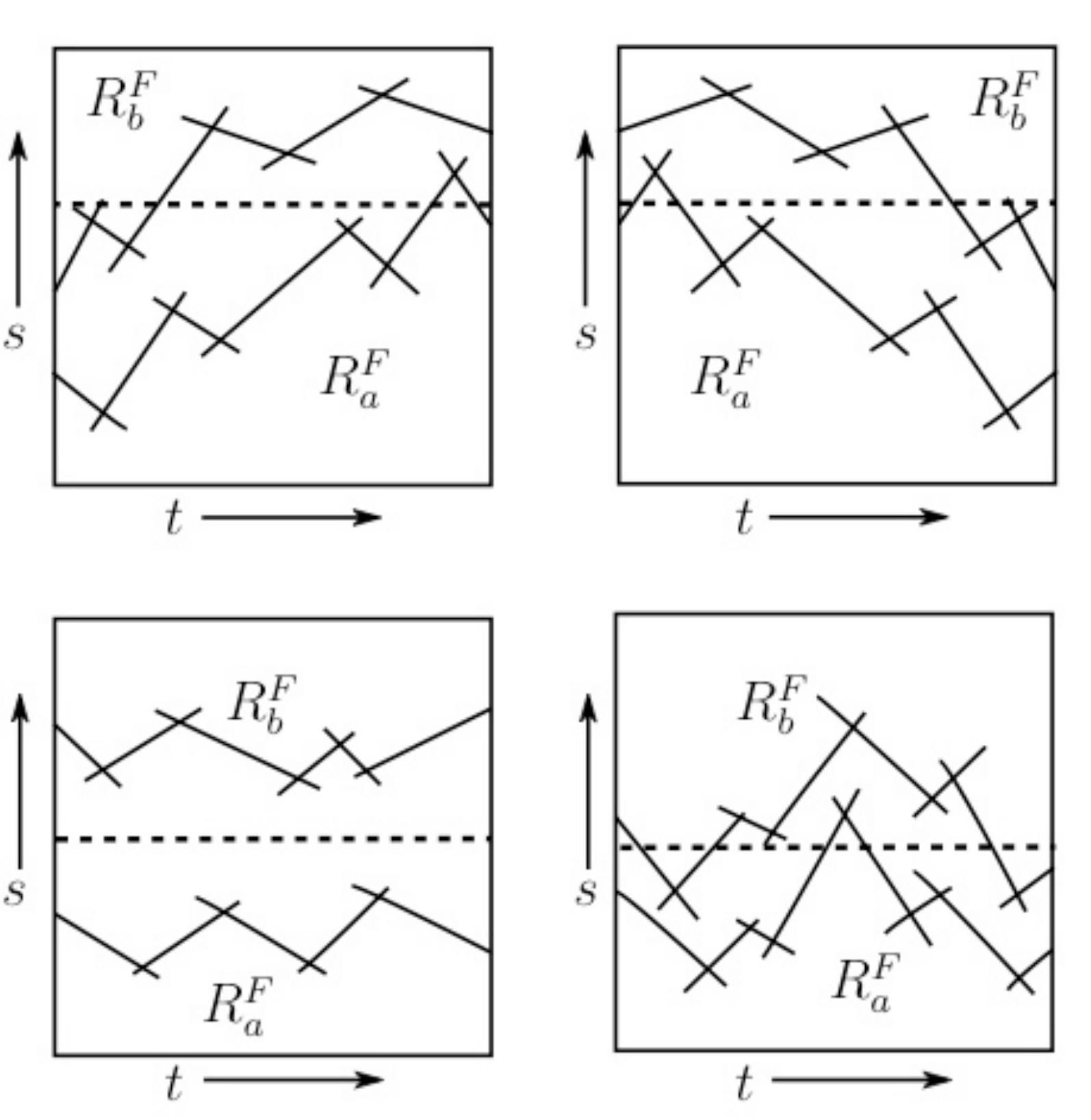}}}
\caption{ Four partial graphics of pairs $f \times h$ of sweepouts. In
  the top two graphics, $h$ $F$-spans $f$. In the bottom left, $h$
  $F$-splits $f$. In the bottom right, $h$ $F$-spans $f$ again but
   the switch between  mostly above and mostly below  happens twice
   along the horizontal arc shown.}
 \label{spanningfig}
\end{figure}

\begin{lemma} \label{R structure}
Let  $f\times h$ be a generic sweepout pair, and let $F$ be an essential subsurface of $\Sigma$. If $F$ is an annulus assume it is not isotopic into $\Lambda$. 

The sets $R^F_a$ and $R^F_b$ are disjoint, open and bounded by arcs of the graphic, so that all interior vertices appearing in $\boundary R^F_a$ or $\boundary R^F_b$ have valence 4. 

Moreover $R^F_a$ and $R^F_b$ intersect each vertical line in a pair of intervals as follows:
For each $t \in (-1, 1)$, there are values $x\le y\in (-1,1)$ such that 
$$ (t,s)\in R_a^F \iff s\in(-1,x) $$
and 
$$ (t,s)\in R_b^F \iff s\in(y,1). $$

\end{lemma}

\begin{proof}
We assume $F$ is a proper subsurface of $S$ (the case $F=\Sigma$ was done in Johnson 
\cite{johnson:flipping}, and our argument is a direct generalization). Openness of $R^F_a$ and $R^F_b$ is clear from the definition. If the sets intersect then  we may select a point $(t,s)$ in the intersection which is 
not in the graphic, so that $\Sigma_t\intersect \Lambda_s$ is a 1-manifold. This 1-manifold divides $\Sigma_t$ into two subsurfaces $X_-=\Sigma_t\intersect V^-_s$ and $X_+ = \Sigma_t\intersect V^+_s$, each of
which can be isotoped off $F_t$. But this is impossible: If $X_-$ can be isotoped off $F_t$ then, after isotopy, 
$X_+$ contains $F_t$ and hence can't be isotoped away from $F_t$, unless $F_t$ is an annulus. 
If $F_t$ is an annulus this means it is isotopic to a neighborhood of the common boundary between $X_+$ and $X_-$, which is isotopic into $\Lambda_s $ contradicting the assumption.  We conclude that $R^F_a$ and $R^F_b$ are disjoint. 

\medskip

Now let $(t,s)\in (0,1)\times(0,1)$ be any point which is  not on the graphic. Again $\Sigma_t$ and 
$\Lambda_s$ are transverse, and hence a small perturbation of $(t,s)$ would yield an isotopic intersection
pattern. In particular a small neighborhood of $(t,s)$ is either contained in $R_a^F$ or in its complement (and similarly for $R_b^F$). It follows that the boundary must be contained in the graphic. 

\medskip

Let $(t,s)\in R_a^F$. Then $\Sigma_t \intersect V^-_s$ can be isotoped out of $F_t$. Since $V^-_s$
monotonically increases with $s$, the same must hold for each $s'<s$. Hence  the set
$\{s:(t,s)\in R_a^F\}$ must, if non-empty, be an interval of the form $(-1,x)$. It is non-empty because, for small enough $s$, $V^-_s$ is a small regular neighborhood of a spine and so intersects $\Sigma_t$ in a union of disks. Similarly  we find that $\{s:(t,s)\in R_b^F\}$ has the form   $(y,1)$. Finally, $x\le y$ follows from the disjointness of $R^F_a$   and $R^F_b$.

It remains to show that only 4-valent vertices appear in the boundary of $R^F_a$ and $R^F_b$. Suppose $v$ is a 2-valent vertex in $\boundary R^F_a$, say. Then there is a neighborhood $U$ of $v$ which intersects the graphic in a pair of arcs incident to $v$, which lie on one side of the vertical line through $v$. Such a pair separates $U$ into two pieces, and hence is equal to $\boundary R^F_a \intersect U$. This contradicts
what we have just proven about the intersection of $R^F_a$ with vertical lines. 

\end{proof}

\subsection*{Relative spanning and splitting}

\medskip

The relations of {\em spanning} and {\em splitting} were introduced in \cite{johnson:flipping}  for pairs of sweepouts.   Here we will extend this notion to spanning and splitting relative to a subsurface $F$.

\begin{definition}\label{ffdef}\rm
We say that $h$ {\it $F$-spans} $f$ if  there is a horizontal arc $[-1,1]\times\{s\}$ in $(-1,1)\times(-1,1)$
that intersects both $R^F_a$ and $R^F_b$. 

\end{definition}

In other words, $h$ $F$-spans $f$ if there are values $s, t_-, t_+ \in [-1,1]$ such that
$\Sigma_{t_-}  \prec_F \Lambda_s $  while $\Sigma_{t_+}  \succ_F \Lambda_s $.   

Figure~\ref{spanningfig} also shows  examples of graphics of pairs of sweepouts that don't span, 
or that span with two distinct such arcs. 

\medskip

The complementary situation is the following: 

\begin{definition}\label{weak splitting}\rm
We say that $h$ {\it weakly $F$-splits} $f$ if there is no horizontal arc $[-1,1]\times\{s\}$ that meets both 
$R^F_a$ and $R^F_b$.

\end{definition}

A somewhat stronger condition which under some circumstances is
equivalent, is the following: 

\begin{definition}\label{strong splitting}\rm
We say that $h$ {\it $F$-splits} $f$  if,  for some $\{s\} \in (-1, 1)$, the arc $[-1,1] \times \{s\}$ is disjoint from the closures of both $R^F_a$ and $R^F_b$.

\end{definition}

We extend these notions to pairs of Heegaard splittings with no sweepout specified: That is, if $\Lambda$ and $\Sigma$ are Heegaard surfaces and there exist some sweepouts $h$ and $f$ representing $\Lambda$ and 
$\Sigma$, respectively, such that $f\times h$ is generic and $h$ $F$-splits $f$ for some $F\subset \Sigma$, we say that $\Lambda$ $F$-splits $\Sigma$; and similarly for $F$-spanning and weak $F$-splitting.

\medskip

\subsection*{Dichotomy}
By definition $F$-spanning and weak $F$-splitting are complementary conditions. Generically, 
$F$-spanning and $F$-splitting are complementary as well, i.e. weak $F$-splitting implies $F$-splitting. The proof proceeds along the lines of \cite{johnson:flipping}, except when $F$ is one of a short list of low-complexity surfaces, when weak $F$-splitting becomes a distinct case:

\begin{lemma}\label{dichotomy}
Let $f$ and $h$ be sweepouts for $M$, and $F$ a connected, essential subsurface of the Heegaard surface 
$\Sigma$ represented by $f$. Suppose that $f\times h$ is generic.

If $F$ is not an annulus,  4-holed sphere or a 1-holed or 2-holed torus, then 
$$ \text{$h$ weakly $F$-splits $f$} \implies \text{$h$ $F$-splits $f$.} $$
Equivalently,  either $h$ $F$-spans $f$, or $h$ $F$-splits $f$. 

In the exceptional cases, if $h$ weakly $F$-splits $f$ but does not $F$-split, then there exists a unique horizontal line $[-1,1]\times \{s\}$ which meets both closures $\bar{R^F_a}$ and $\bar {R^F_b}$ in a single point, which is a vertex of the graphic. 

\end{lemma}

\begin{proof}
By Lemma \ref{R structure} the
set of heights of horizontal lines meeting $R^F_a$ has the form
$(-1,x)$ and the set of heights of horizontal lines meeting $R^F_b$
has the form $(y,1)$. If $h$ does not $F$-span $f$, then these sets
are disjoint so $x \le y$. If $x<y$ then any line with height
$s\in(x,y)$ misses both $\bar R^F_a$  and $\bar R^F_b$, so that $h$
$F$-splits $f$, and we are done. 

If $x=y$ then the line $[-1,1]\times\{x\}$ meets a maximum-height point of 
$\bar R^F_a$ as well as a minimum-height point $\bar R^F_b$. Such a
point lies on the graphic by Lemma \ref{R structure} and it must be a
vertex, for the only other possibility is an interior horizontal
tangency of the graphic. Such a tangency corresponds to a critical
point of $h$ away from the spine, which is ruled out by definition of
a sweepout (see \cite{johnson:stable} for details).

Since $f\times h$ is generic, a horizontal line can only meet one
vertex, so the maximum of $\bar{R^F_a}$ and the minimum of
$\bar{R^F_b}$ are the same vertex $(t,x)$ of the graphic. It remains
to show that this can only happen in the given exceptional cases. 

The fact that $(t,x)$ is a vertex (of valence 4 by Lemma \ref{R structure})  means that the function 
$h_t \equiv h|_{\Sigma_t}$ has two critical points on  $\Gamma \equiv h_t^{-1}(x)$. Choose 
$x_- < x_+$ so that $x$ is the unique critical value in $[x_-,x_+]$, and  let 
$$Q=h_t^{-1}([x_-,x_+]),$$ 
which is a regular neighborhood of $\Gamma$ in $\Sigma_t$. Since $(t,x_-)\in R_a^F$, the subsurface 
$Y_- \equiv h_t^{-1}([-1,x_-])$ can be isotoped out of $F_t$.  Since $(t,x_+)\in R_b^F$, the subsurface 
$Y_+\equiv h_t^{-1}([-1,x_-])$ can also be isotoped out of $F_t$. Equivalently $F_t$ can be isotoped out of
each of $Y_-$ and $Y_+$, and it follows that this can be done simultaneously (this is a general property of surfaces). Therefore after this isotopy we find that $F_t \subset Q$. 

However $Q$ cannot be very complicated -- since $h_t^{-1}(x)$ has two critical points, $Q$
is a disjoint union of annuli, disks, 3-holed spheres and possibly one component that is a 4-holed sphere
or 2-holed torus.  (The detailed analysis of possibilities for $Q$, which we will need in Section \ref{boundsect}, will be done in Section \ref{saddle transitions}.) The only essential subsurfaces of $\Sigma_t$ that can embed in $Q$ are therefore annuli, 3-holed spheres, 4-holed spheres, and 1- and 2-holed tori. This takes care of
all the exceptional cases. 

\end{proof}

\section{Spanning implies coherence}
\label{spanning case}

In this section we consider the $F$-spanning case, for which we can show that the surface $F$ is 
isotopic into $\Lambda$. 

\begin{proposition}\label{spanning is good}
Let $\Sigma$ and $\Lambda$ be Heegaard surfaces in a 3-manifold $M$
and let $F\subset\Sigma$ be a proper connected essential
subsurface. If $\Lambda$ $F$-spans $\Sigma$ 
then after isotoping $\Lambda$ we obtain a surface
whose intersection with $\Sigma$ contains $F$.

\end{proposition}

\begin{proof}
Let $f$ and $h$ be sweepouts representing  $\Sigma$ and $\Lambda$, respectively, such that $h$ 
$F$-spans $f$ and $f\times h$ is generic.  Hence there is a level surface $\Lambda_s$ and values 
$t_- , t_+ \in (-1, 1)$  such that  $\Sigma_{t_{-}}\prec_F \Lambda_s $  and 
$\Sigma_{t_{+}} \succ_F \Lambda_s $.  We may assume without loss of generality that $t_- < t_+$.

We may identify the complement of the spines of $f$, namely
$f^{-1}((-1,1))$, with $\Sigma\times(-1,1)$ in a way that is unique up
to level-preserving isotopy.  Consider the sub-manifold $F \times
J \subset M$, where $J=[t_-,t_+]$.  Because $\Sigma_{t_{-}}\prec_F \Lambda_s$,
the handlebody $V^+_s= h^{-1}([s,1])$ intersects $\Sigma\times\{t_-\}$ in a set that
can be isotoped outside of $F\times\{t_-\}$; equivalently, the
subsurface $F \times \{t_{-}\}$ can be isotoped within $\Sigma \times
\{t_{-}\}$ so that it is contained in $V^-_s$.  Similarly, since
$\Sigma_{t_{+}} \succ_F \Lambda_s $, we can isotope the surface $F
\times \{t_{+}\}$ inside $\Sigma \times \{t_{+}\}$ so that it is
contained in $V^+_s$. 
After a level-preserving isotopy 
of $\Sigma \times (-1, 1)$,  we may therefore assume that
$F\times\{t_-\}$ and $F\times\{t_+\}$ are contained in $V^-_s$ and
$V^+_s$, respectively. 
The surface $ S = \Lambda_s \cap F \times J$ therefore separates $F
\times \{t_-\}$ from $F \times \{t_+\}$ within $F\times J$, since 
$\Lambda_s$ separates $V^-_s$ from $V^+_s$. 

Note that a surface in $F\times J$  with boundary in $\boundary F
\times J$ separates $F\times\{t_-\}$ from $F\times\{t_+\}$ if and
only if its homology class in $H_2(F\times
J,\boundary F\times J)$ is nonzero.

Compressing $S$ if necessary, we obtain an incompressible surface in the same homology class. Let
$S'$ be a connected component which is still nonzero in $H_2(F\times J,\boundary F\times J)$. Then $S'$ separates $F\times\{t_-\}$  from  $F\times\{t_+\}$ and hence (up to orientation) must be homologous to
$F\times\{t_-\}$. The vertical projection of $S'$ to $F\times\{t_-\}$ is therefore a proper map which is $\pi_1$-injective and of degree $\pm 1$. By a covering argument it must also be $\pi_1$-surjective, and hence
Theorem 10.2 of ~\cite{He} implies that $S'$ is isotopic to a level surface $F\times\{t\}$. 

Making the isotopy ambient and keeping track of the 1-handles corresponding to the compressions that gave us $S'$, we can obtain an isotopic copy of $\Lambda$ which contains $F\times\{t\}$ minus attaching disks for the 1-handles. Now we can slide these disks outside of $F\times\{t\}$. Hence $\Lambda$ itself
is isotopic to a surface containing $F\times \{t\}$, as claimed.

\end{proof}

\section{Saddle transitions} \label{saddle transitions}

In this section we examine more carefully the intersections $ \Sigma_t \intersect \Lambda_s $, and the relationship between their regular neighborhoods in the two surfaces.

Fix $(t,s)$ for the rest of this section. The interesting case is  when $(t,s)$ lies in the graphic, and
hence the surfaces are not transversal, or equivalently, $t$ is a critical value of $f|_{\Lambda_s}$ and $s$ is a critical value of $h|_{\Sigma_t}$. As in the proof of Lemma \ref{dichotomy}, take an interval $[s_-,s_+]$
in which $s$ is the only critical value of $h|_{\Sigma_t}$ (if any), and let  
$$Q=(h|_{\Sigma_t})^{-1}([s_-,s_+]) \subset \Sigma_t.$$ 
Similarly choose an interval $[t_-,t_+]$ containing $t$ with no other critical values of $f|_{\Lambda_s}$, and
define
$$ Z = (f|_{\Lambda_s})^{-1}([t_-,t_+]) \subset \Lambda_s. $$
Then $Q$ is a regular neighborhood of $\Gamma \equiv \Sigma_t \intersect \Lambda_s $ in $\Sigma_t$, and $Z$ is a regular neighborhood of $\Gamma$ in $\Lambda_s$. 

To compare $Q$ and $Z$, fix an identification of $f^{-1}((-1,1))$ with $\Sigma\times(-1,1)$
such that $f$ becomes projection to the second factor, and let
$$p:f^{-1}((-1,1)) \to \Sigma$$ 
denote the projection to the first factor.  For convenience identify $\Sigma$ with $\Sigma_t$ so that 
$p:\Sigma_t \to \Sigma$ can be taken to be the identity. Now we can compare $p(Z)$ to $Q$ within $\Sigma$.

The trivial case is that of a component of $\Gamma$ that contains no critical points; in this case the components of $Q$ and $Z$ that retract to it are both annuli and the restriction of $p$ is  homotopic to a diffeomorphism between them. Let us record this observation: 

\begin{lemma} \label{parallel levels}
Two level curves of $f|_{\Lambda}$ which are not separated by critical points have homotopic $p$-images in $\Sigma$.
 
\end{lemma}

\qed

\medskip

Another simple case is of a component of $\Gamma$ which is an isolated (hence central) singularity, and the corresponding components of $Q$ and $Z$ are disks. 

If $\Gamma_0$ is a component containing one saddle singularity, then it is a figure-8 graph, and
both associated components $Q_0\subset Q$ and $Z_0\subset Z$ must be  3-holed spheres, or pairs of pants. In this case we have:

\begin{lemma}\label{pants levels}
If $\Gamma_0$ is a figure-8 component of $\Gamma$ then  $p|_{Z_0}$ is homotopic to a diffeomorphism 
$Z_0\to Q_0$. In particular $p(\boundary Z_0)$ is homotopic to a collection of simple disjoint curves in 
$\Sigma$. 

\end{lemma}

Recall that for any point of the graphic that is not a vertex, $\Gamma$ contains exactly one singularity, so if it is a saddle we have exactly one such figure-8 component of $\Gamma$.

\begin{proof}
Because each boundary curve of $Z_0$ lies on a level surface $\Sigma_{t_-}$ or $\Sigma_{t_+}$, it is embedded by $p$. On the other hand it is homotopic to an essential curve in $\Gamma_0$, and since $p$ is
continuous the same is true for its image in $\Sigma$.  The only essential simple curves in a 3-holed sphere are parallel to boundary components, so $p(\boundary Z_0)$ is homotopic to $\boundary Q_0$. The lemma follows. 

\end{proof}

The most complicated case is that of  a component $\Gamma_0\subset \Gamma$ that contains two saddle singularities. There are only two possible isomorphism types for $\Gamma_0$ as a graph with two vertices of valence 4,  and a total of three ways that $\Gamma_0$ can embed as a {\em level set} of a function with nondegenerate singularities. These are indicated in Figure \ref{Qcases}, which applies to both $Q_0$ and $Z_0$. 

\begin{figure}[htp]
\centerline{\epsfbox{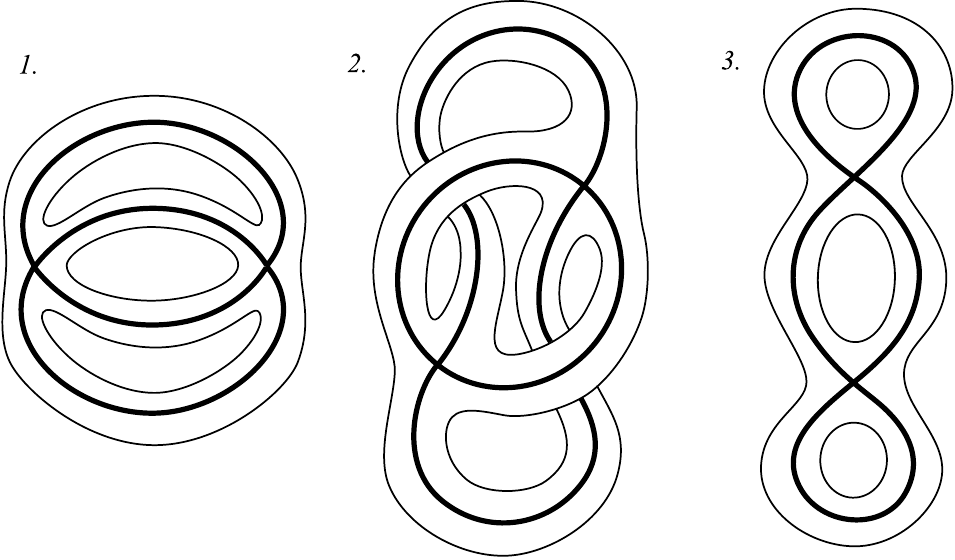}}
\caption{The three possible types for $Q_0$ or $Z_0$, depicted
  as immersed surfaces in the plane with the level set $\Gamma_0$ in heavy
  lines. The first and third are 4-holed spheres and the second is a
  2-holed torus.}
\label{Qcases}
\end{figure}

To understand how $p$ looks in this case,  note first that at the vertices of $\Gamma_0$ the surfaces are
tangent so $p$ is a local diffeomorphism. Hence it either preserves or reverses orientation at these points. This determines the local picture of $Q_0$ and $Z_0$ at each tangency, and the rest of the configuration is determined by how $Z_0$ attaches along the edges of $\Gamma_0$. Either the two orientations match or they do not. When they match, one sees that  $p|_{Z_0}$ is homotopic to a diffeomorphism -- we call this the untwisted case. If they do not match -- the twisted case -- then in fact $Q_0$ and $Z_0$ may not even be diffeomorphic. The three twisted cases are depicted in Figure \ref{QandZ}, where in the first case $Q_0$ is
a 4-holed sphere and $Z_0$ a 2-holed torus, in the middle case they change roles, and in the third case both are 4-holed spheres but $p$ is not homotopic to a diffeomorphism.

\begin{figure}[htp]
\centerline{\epsfbox{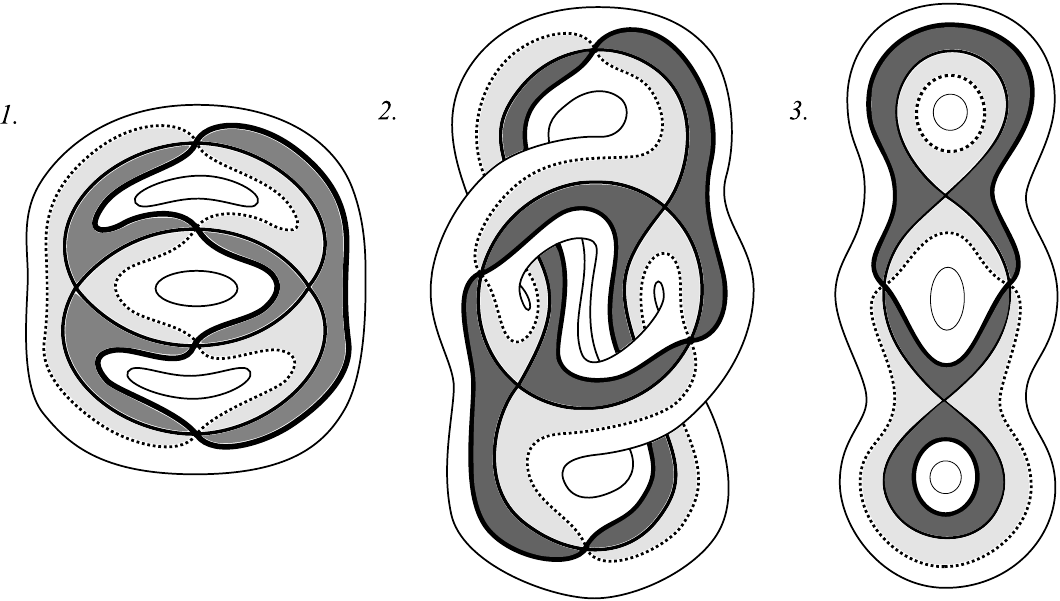}}
\caption{The twisted cases of $Q_0$ and $Z_0$. In each case $Q_0$
  is depicted immersed in the plane, while $Z_0$ intersects it. The
  parts of $Z_0$ above the plane are shaded darker 
  than the parts below. $\boundary_+Z_0$ is a heavy line,
  $\boundary_-Z_0$ is dotted. The map $p$ is essentially projection to the
  plane.}
\label{QandZ}
\end{figure}

Note as in Lemma \ref{pants levels}  the boundary curves of $Z_0$ are mapped to simple curves; however in the twisted case their images are not disjoint. Their intersections are prescribed by Figure \ref{QandZ}, and we wish to record something about how they look from the point of view of our subsurface $F$. In fact the only time that this double-saddle case actually occurs is when $h$ weakly $F$-splits $f$, and $(t,s)$ is the vertex in the
intersection of  $\bar{R^F_a}$ and $\bar{R^F_b}$, as discussed in Lemma \ref{dichotomy}. In this case $F$ (up to isotopy) lies in $Q_0$, which allowed us  to  conclude in Lemma \ref{dichotomy} that it must be in one of the exceptional cases. 

Now divide up $\boundary Z_0$ as  $\boundary_+Z_0 = Z_0 \intersect f^{-1}(t_+)$ and
$\boundary_-Z_0 = Z_0 \intersect f^{-1}(t_-)$, each of which is embedded by $p$. 

\begin{lemma}\label{Q and Z}
Suppose that $h$ weakly $F$-splits $f$, $(t,s)$ is the vertex comprising $\bar{R^F_a} \intersect \bar{R^F_b}$, and $\Gamma_0$ is the component of $\Gamma = \Sigma_t\intersect \Lambda_s$ whose regular neighborhood $Q_0$ in $\Sigma_t$ contains $F$. Let $Z_0$ be the regular neighborhood of $\Gamma_0$ in
$\Lambda_s$.

If two components $\alpha_+\subset p(\boundary_+ Z_0)$ and $\alpha_-\subset p(\boundary_- Z_0)$
intersect $F\subset Q_0$ essentially, then 
$$ d_F(\alpha_+,\alpha_-) \le 3. $$

\end{lemma}

\begin{proof}
In the untwisted case all boundary components of $Z_0$ map to boundary components of $Q_0$, and hence are peripheral and the lemma holds vacuously. The same occurs (via Lemma \ref{pants levels}) if $Q_0$ is
a 3-holed sphere and $F$ is an annulus.  We therefore consider the twisted cases from now on. 

Consider first the case that $F$ is actually isotopic to $Q_0$.  Note that $F$ can only be a proper subsurface of $Q_0$ if $F$ is an annulus, or if $F$ is a 1-holed torus and $Q_0$ is a 2-holed torus.  We return to these cases in the end. 

We can read off $\alpha_+$ and $\alpha_-$ from the diagrams in Figure \ref{QandZ}: they are components of the thickened curves and the dotted curves, respectively. In case (1), they intersect each other four times, and one can see from Figure \ref{twoarcs}(a) that they are connected by a chain of length 3 in which the two middle vertices are the arcs $b^\pm$ in the figure. In case (2) each possible $\alpha_-$ intersects each possible 
$\alpha_+$ exactly once, and it is easy to see that the distance is 2. In case (3) they intersect twice (here two components are actually inessential in $Q_0$, so we must consider the others), and again the distance is 3. 

\begin{figure}[htp]
\centerline{\epsfbox{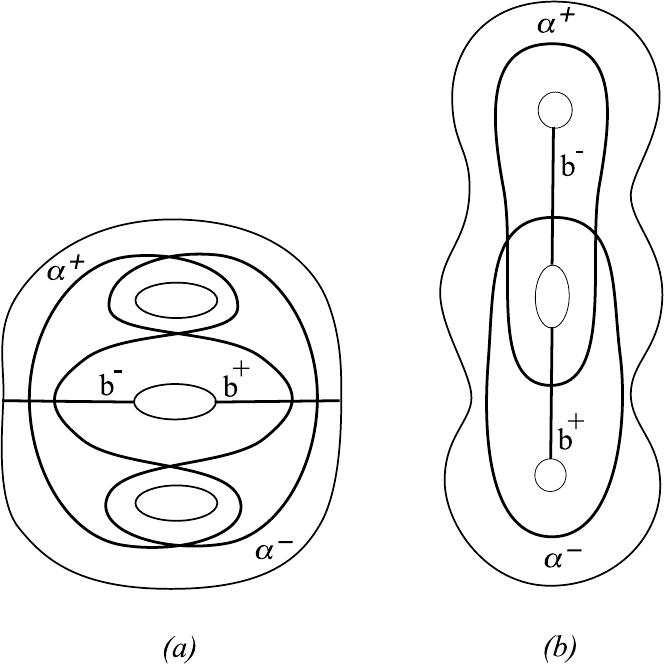}}
\caption{Checking that $d_{\AC(Q_0)}(\alpha^+,\alpha^-) = 3$}
\label{twoarcs}
\end{figure}

Consider now the case that $Q_0$ is a 2-holed torus -- case (2) of Figure \ref{QandZ} -- and $F$ is a 1-holed torus.  Again $\alpha_+$ and $\alpha_-$ intersect exactly once, therefore their intersections with $F$ (if any)
intersect at most once. Hence their distance in $\AC(F)$ is  at most 2. 

Finally we consider the case that $F$ is an annulus. We can assume $F$ is not boundary-parallel in $Q_0$, because otherwise it would not intersect $\boundary Z_0$ essentially. The possibilities for $Q_0$ are again enumerated in Figure \ref{QandZ}. 

In case (1), let $(\alpha^+,b^+,b^-,\alpha^-)$ be the sequence from Figure \ref{twoarcs}(a). We claim that if $F$ intersects $\alpha^+$ and $\alpha^-$ essentially then it must intersect $b^+$ and $b^-$ essentially as well. This is because if $F$ misses $b^-$ it must be parallel to $\alpha^-$, and similarly for $b^+$ and $\alpha^+$. 
This will imply that the same sequence gives a distance bound of 3 in $\AC(F)$, provided we understand how to lift the arcs $b^\pm$ to the annulus complex. 

\begin{figure}[htp]
\centerline{\epsfbox{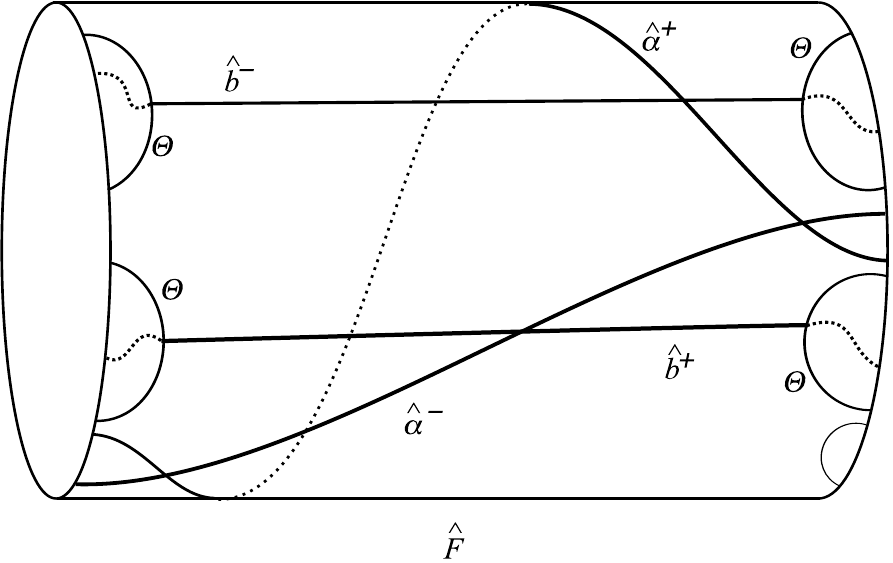}}
\caption{Lifting curves and arcs to the annulus $\hat F$.}
\label{theta}
\end{figure}

The boundary components of $Q_0$ lift to a union $\Theta$  of boundary-parallel arcs in the annulus $\hat F$ (see Figure \ref{theta}). The arcs $b^\pm$ lift to disjoint arcs connecting  components of $\Theta$, and since they cross $F$ essentially, we can choose lifts with endpoints on  arcs of $\Theta$ adjacent to opposite sides of
$\hat F$. Given such a lift, we can join its endpoints to $\boundary \hat F$ using arcs in the disks bounded by the components of $\Theta$ that it meets. This results in disjoint properly embedded arcs $\hat b^\pm$, such that $\hat b^+$ is disjoint from the lifts of $\alpha^+$ because the latter never cross $\Theta$, and similarly for
$\alpha^-$. We conclude that $d_{\AC(F)}(\alpha^+,\alpha^-)\le 3$. 

In case (2), Figure \ref{QandZ} shows that any choice of $\alpha^+$ and $\alpha^-$ intersect exactly
once. If they intersect $F$ essentially, by the argument given in \cite[\S2.1]{farb-lubotzky-minsky}
their lifts to $\hat F$ cross at most twice.  We conclude that the distance in $\AC(F)$ is at most 3. 

In case (3) we can use exactly the same argument as for case (1), using the arcs in Figure \ref{twoarcs}(b).

\end{proof}

\section{Splitting bounds distance}
\label{boundsect}

In this section we put together the results obtained so far to show that $F$-splitting gives a bound on 
$d_F(\Sigma)$. 

\begin{proposition}
\label{neitherfacinglem}
Let $\Sigma$ and $\Lambda$ be Heegaard surfaces for $M$, and let
$F\subset\Sigma$ be a proper connected essential subsurface. If $\Lambda$ 
$F$-splits $\Sigma$ then 
$$d_F(\Sigma) \le 2g(\Lambda) .$$
If $\Lambda$ weakly $F$-splits $\Sigma$ then
$$d_F(\Sigma) \le 2g(\Lambda) + 2.$$
\end{proposition}

\begin{proof}
Choose sweepouts $h$ and $f$ with $f\times h$ generic, such that $h$ $F$-splits $f$ or weakly $F$-splits $f$, according to our hypothesis. 

In the $F$-splitting case there is a value $s$ such that $[-1,1]\times\{s\}$ is disjoint from the closures of 
$R_a^F$ and $R_b^F$. In particular this means that there is an interval of such values and we may choose one such that $[-1,1]\times\{s\}$ meets no vertices of the graphic. It follows that the function $f|_{\Lambda_s}$ is a Morse function, having at most one critical point per critical value.  We fix this  $s$ and henceforth identify $\Lambda$ with $\Lambda_s$. 

In the weak $F$-splitting case there is a value of $s$ such that $[-1,1]\times\{s\}$ intersects the closures of $R_a^F$ and  $R_b^F$ only at their unique intersection point, which is a vertex $(t,s)$ that we will call the 
{\em weak splitting vertex}. We again fix this $s$ and let $\Lambda = \Lambda_s$. 

The proof proceeds along the following lines:  The intersections $\Lambda \intersect \Sigma_t$ are the level sets of $f|_\Lambda$, and the idea is to consider them as curves on $\Sigma$, argue using the $F$-splitting property that they intersect $F$ essentially, and then use the topological complexity of $\Lambda$ to
bound the diameter in $\AC(F)$ of the corresponding set. This kind of argument originates with Kobayashi 
\cite{kobayashi:heights}. It was extended in various ways by Hartshorne \cite{hartshorn}, Bachman-Schleimer \cite{bachman-schleimer:bundles}, Scharlemann-Tomova \cite{scharlemann-tomova:genusbounds} and Johnson \cite{johnson:flipping}. Our argument is a direct extension of the one used in \cite{johnson:flipping}.

We use the map $p:f^{-1}((-1,1)) \to \Sigma$, as in Section \ref{saddle transitions}, to map level curves of $f|_{\Lambda}$ into $\Sigma$. The lemmas in Section \ref{saddle transitions} will help us control what happens as we pass through critical points of $f|_{\Lambda}$. 

One important complication is the possibility that curves that are essential in one surface are trivial in the other. To handle this we adapt an argument of \cite{johnson:flipping}:

\begin{lemma}\label{interval}
If  $d_F(\Sigma) > 3$, then there is some non-trivial interval  $[u,v] \subset (-1,1)$ such that for each regular 
$t \in [u,v]$, every loop of $\Sigma_t \cap \Lambda$ that is trivial in  $\Lambda$ is trivial in $\Sigma_t$. 

\end{lemma}

Here and below when we say $t$ is regular we mean it is a regular
value for $f|_\Lambda$.

\begin{proof}  
As in Lemma \ref{parallel levels}, the isotopy classes in $\Sigma$ 
of level sets of $f|_{\Lambda}$ are constant over an
interval between two critical values. 
Thus it suffices to find a single regular value $t$ such that every
loop of $\Sigma_t \cap \Lambda$ that is trivial in $\Lambda$ is
trivial in $\Sigma_t$.

Note also that if $\Sigma\ssm F$ is compressible to at least one side of
$\Sigma$ then $d_F(\Sigma) \le 3$ by Lemma \ref{disks everywhere}.
Hence we may assume that every disk in $H^+$ or $H^-$ intersects $F$.  

Assume, seeking a contradiction, that for every regular $t$ there is a loop of
$\Sigma_t \cap\Lambda$ that is essential in $\Sigma_t$ and trivial
in $\Lambda$. We can assume this loop is innermost within $\Lambda$
among intersection loops that are essential in $\Sigma_t$. Within the
disk that it bounds in $\Lambda$ there might still be intersection
loops that are trivial in both surfaces, but after some disk exchanges
we see that we have a loop
of $\Sigma_t \cap \Lambda$ which bounds an essential disk in $H^-_t$
or $H^+_t$.  The loops of $\Sigma_t \cap \Lambda$, for regular
values of $t$, are pairwise disjoint in $\Sigma_t$, so if there are
loops of $\Sigma_t \cap \Lambda$ that bound disks on opposite sides
of $\Sigma_t $ then $d_F(\Sigma) \leq 1$. Note we are using here the
fact that all such disks must meet $F$ essentially.

If $\Lambda$ $F$-splits $\Sigma$ and $s$ has been chosen as above,
then for a non-regular value $t_0$, there is a unique critical
point. If it is a saddle point then
Lemma \ref{pants levels} tells us that
the $p$-images of loops of intersection of $\Sigma_t \cap \Lambda$ for the
regular values $t$ just before $t_0$ and just after $t_0$ are also
pairwise disjoint in $\Sigma$. We call this a {\em saddle
  transition}. So if the loops switch from bounding a disk on one
side of $\Sigma_{t_0} $ to the other, we again find that the distance
$d_F(\Sigma)$ is at most one. If the point is a central singularity
then (again referring to \S\ref{saddle transitions}) it corresponds to
the appearance or disappearance of a level curve which is trivial in {\em
  both} $\Lambda$ and $\Sigma$, and so can be ignored. We call this a
{\em central transition}.

If $\Lambda$ {\em  weakly} $F$-splits $\Sigma$ and $s$ has been chosen as above,
then a non-regular value $t_0$ could occur at the weak splitting
vertex, $(t_0,s)$. 
In this case, recall that $F$
is in one of the exceptional cases, and by Lemma \ref{Q and Z}, any
level curves just before and just after $t_0$ which intersect $F$
essentially have distance at most 3 in $\AC(F)$. 
Hence we obtain $d_F(\Sigma) \le 3$ if these loops
correspond to disks on opposite sides. 

Now, for $t$ near $-1$, all the loops of
intersection bound disks on the negative side of $\Sigma_t$ and for
$t$ near $1$, they all bound disks on the positive side. Hence they
must switch sides at some point, thus implying $d_F(\Sigma) \leq 3$.  We
conclude that there is some regular value $t$ for which the assumption
does not hold. Thus the first paragraph gives the desired interval
$[u, v]$.

 \end{proof}

\medskip

For a regular value $t\in (-1,1)$ of $f|_\Lambda$, let $\LL_t$ denote the set of
nontrivial isotopy classes in
$\Sigma$ of the  $p$-images of the curves of
$(f|_\Lambda)^{-1}(t)$.
 For an interval $J\subset (-1,1)$ let
$\LL_J$ denote the union of $\LL_t$ over regular $t\in J$. 

We now observe that
for each regular $t\in (-1,1)$, $\LL_t$ contains a curve
that intersects $F$ essentially, and in particular $\pi_F(\LL_t)$ is
nonempty. 
This is because both $F$-splitting and weak $F$-splitting tell us that
$(t,s)$ is in neither $R^F_a$ nor $R^F_b$ (it cannot be the weak
splitting vertex since $t$ is regular), so that 
both $V^+_s\intersect
\Sigma_t$ and $V^-_s\intersect \Sigma_t$ cannot be isotoped off $F_t$
in $\Sigma_t$,
and hence neither can their common boundary.

We can now bound the diameter in $\AC(F)$ of $\LL_{[u,v]}$ for $[u,v]$
satisfying the conclusion of  Lemma \ref{interval}.

\begin{lemma}\label{L diam bound}
Let $[u,v]$ be an interval 
such that for each regular $t \in [u,v]$, every loop of
$\Sigma_t \cap \Lambda$ that is trivial in  $\Lambda$ is trivial
in $\Sigma_t$. If $(u,v)\times\{s\}$ encounters no vertices of the
graphic then
$$
\diam_{F}(\LL_{[u,v]}) \le 2g(\Lambda)-2.
$$
If $(u,v)\times\{s\}$ meets the weak splitting vertex then
$$
\diam_{F}(\LL_{[u,v]}) \le 2g(\Lambda)-1.
$$
\end{lemma}

\begin{proof}
The critical values of $f|_\Lambda$ cut $[u,v]$ into intervals. Let
$t_0<\cdots<t_n\in(u,v)$ be a selection of one regular point for each
interval. Then $\union_i\LL_{t_i} = \LL_{[u,v]}$. 

Consider the transition from $\LL_{t_i}$ to $\LL_{t_{i+1}}$. If it is
a central transition then only the loss or gain of a curve trivial in
both $\Lambda$ and $\Sigma$ is involved.  Hence
$\LL_{t_i} = \LL_{t_{i+1}}$. 

Suppose it is a saddle transition (as in the proof of Lemma \ref{interval}),
but that at least one of the curves 
of the associated 3-holed sphere is trivial in $\Lambda$. Then its $p$-image
is trivial in $\Sigma$, and the remaining two curves are trivial or
homotopic. Once more $\LL_{t_i} = \LL_{t_{i+1}}$.

If it is an {\em essential} saddle transition, meaning all three
boundary curves are essential in $\Lambda$, then Lemma \ref{pants levels}
implies that all the curves of $\LL_{t_i}$ and $ \LL_{t_{i+1}}$ are
pairwise disjoint. 

Suppose we are in the first case, where $(u,v)\times\{s\}$ encounters
no vertices of the graphic. Then these transitions are the only possibilities.

Since every $\LL_{t_i}$ intersects $F$ essentially we find that 
each $\mu_i \equiv \pi_F(\LL_{t_i})$ is nonempty, and two successive
$\mu_i$ are equal unless they differ by an essential saddle, in which
case they are at most distance 1 apart. Since the
pants that occur among level sets of $f|_{\Lambda}$ are all disjoint,
and since the number of disjoint essential pants in $\Lambda$ is
bounded by $-\chi(\Lambda) = 2g(\Lambda) - 2$, this gives the desired 
bound on $\diam(\pi_F(\LL_{[u,v]}))$.

If $(u,v)\times\{s\}$ meets the weak splitting vertex then, as in the 
proof of Lemma \ref{interval}, we apply Lemma \ref{Q and Z} to show
that the curves before and after have
distance at most $3$ in $\AC(F)$. On the other hand such a transition uses up
two saddles. If we can show that both are essential in $\Lambda$, this
will make our count higher by
just 1, i.e. we get a bound of  $2g(\Lambda)  - 1$. We therefore
proceed to show this. 

Note first that only the twisted case can
occur here: If $Q_0$ and $Z_0$ are in the untwisted case then in fact $p(\boundary
Z_0)$ is isotopic to $\boundary Q_0$ and hence does not intersect
$F$. This means that for the level $t_i$ just after (or before) the
weak splitting vertex, $\LL_{t_i}$ does not intersect $F$, which is a
contradiction. Hence this case cannot happen. 

Recall that in the weak splitting case $F$ is isotopic into $Q_0$ (notation as in Section
\ref{saddle transitions}).  Suppose $Q_0$ is an essential subsurface of
$\Sigma$ (which holds in particular whenever 
$F$ is isotopic to $Q_0$, since $F$ itself is essential).
By the hypothesis on $[u,v]$
the boundary curves of $Z_0$ cannot be trivial in $\Lambda$ since they
map to nontrivial curves in $Q_0$. Hence $Z_0$ is essential in
$\Lambda$ so both saddles are essential.

Suppose that $Q_0$ is not essential in $\Sigma$, and hence $F$ is a
proper subsurface of $Q_0$. If $Q_0$ is a 4-holed sphere, as in cases
(1) and (3) of Figure \ref{QandZ}, then $F$ must be an annulus. Since
at least one of the boundary components of $Q_0$ is inessential in $\Sigma$,
$F$ must be  isotopic to one of the remaining essential boundary components of
$Q_0$. This means that $F$ has no essential intersections with
$\boundary Z_0$, which as above contradicts the fact that $\LL_{t_i}$
always intersects $F$ essentially.

If $Q_0$ is a 2-holed torus, as in case (2) of Figure \ref{QandZ}, 
then we observe that none of the boundary curves of $Z_0$ can be
inessential in $\Sigma$, since each one has exactly one intersection
point with one of the others. Hence they must be nontrivial in
$\Lambda$ as well, so again we get two essential saddles, and therefore
obtain a bound of $2g(\Lambda)-1$.

\end{proof}

We are now ready to complete the proof of Proposition \ref{neitherfacinglem}.

\medskip

If $d_F(\Sigma) \leq 3$ then the claim of Proposition
\ref{neitherfacinglem} follows immediately, since  we assumed
$\Lambda$ has genus at least two. Hence we may assume that
$d_F(\Sigma) > 3$. 

In particular we can obtain at least one interval $[u,v]$ satisfying
the conclusion of Lemma \ref{interval}.
Note that if two such intervals intersect then
their union also satisfies the conclusion of Lemma
\ref{interval}. Hence we can let $[u,v]$ be a {\em maximal} such
interval. The endpoints $u$ and $v$ must be critical values, else we
could enlarge the interval. Note also  $u > -1$, since for $t$
close to $-1$, $\Sigma_t$ is the boundary of a small neighborhood
of the spine $f^{-1}(-1)$ and so intersects $\Lambda$ in small circles
that bound disks in $\Lambda$ but are essential in
$\Sigma_t$. Similarly we see $v < 1$. Let $u'<u$ be a regular value not
separated from $u$ by any critical values. Then $\Lambda \intersect \Sigma_{u'}$ must
contain a component $\beta$ that is trivial in $\Lambda$ but essential in
$\Sigma_{u'}$. 

Moreover $\beta$ must in fact be trivial in the handlebody 
$f^{-1}([-1,u'])$ -- in other words $\beta$ determines an element of $\DD(H^-)$.
To see this, let $E$ be the disk in $\Lambda$ bounded by $\beta$. Let
$u''\in (u,v)$ be a regular value, and let
$E_1= E\intersect f^{-1}([-1,u''])$. Then $E_1$ contains a
neighborhood of $\beta$, and any internal boundary component of $E_1$ is
trivial in $E$, and hence in $\Lambda$. Since $u''\in[u,v]$, such a
curve is also trivial in $\Sigma_{u''}$, and so after a surgery and isotopy we can
obtain a disk in $f^{-1}([-1,u'])$, whose boundary is $\beta$. 

Similarly, if we
consider a regular $v'>v$ not separated from $v $ by critical values
we obtain $\beta'\in\DD(H^+)$ in the level set of $v'$.

As before, since $d_F(\Sigma) > 3$ we may assume by 
Lemma \ref{disks everywhere} that both $\beta$ and $\beta'$ intersect
$F$ essentially.

Now assume we are in the $F$-splitting case, so that all critical
values have single critical points.  In that case, 
by Lemma \ref{pants levels}, $\beta$ is disjoint from the regular
curves at level just above $u$, and $\beta'$ is disjoint from the
regular curves at level just below $v$. It follows that
$d_F(\beta,\LL_{[u,v]}) \le 1$, and similarly for $\beta'$. We
conclude via Lemma \ref{L diam bound} that $d_F(\beta,\beta') \le
2g(\Lambda)-2+2$, which is what we wanted to prove.

In the weakly $F$-splitting case, there could be one critical value
corresponding to a vertex of the graphic. If this occurs inside
$(u,v)$ then we have already taken account of it in the bound of
$2g(\Lambda)-1$ in Lemma
\ref{L diam bound}, so we get a final bound of 
$d_F(\beta,\beta') \le 2g(\Lambda)+1$. If the vertex
occurs at $u$ or at $v$, then we get contribution of $3$ instead
of 1 from one of the disks. Lemma \ref{L diam bound} gives us $2g(\Lambda)-2$ in that case,
so we get $d_F(\beta,\beta') \le 2g(\Lambda)-2+4 =
2g(\Lambda)+2$. That concludes the proof of Proposition \ref{neitherfacinglem}

\end{proof}

\section{Proof of the main theorem}
\label{proof main}

\begin{proof}[Proof of Theorem~\ref{mainthm}]
Let $f$ and $h$ be sweepouts for $\Sigma$ and $\Lambda$, respectively.
Isotope $f$ and $h$ so that $f \times h$ is generic. 

Suppose first that $F$ is not an annulus, 4-holed sphere, or 1- or
2-holed torus. If $d_F(\Sigma) > 2g(\Lambda)$ then, by Proposition
\ref{neitherfacinglem}, $h$ cannot $F$-split $f$. By
Lemma \ref{dichotomy}, if $h$ does not $F$-split $f$ then $h$ must
$F$-span $f$, and then Proposition \ref{spanning is good} implies that
after isotopy the
intersection of $\Lambda$ and $\Sigma$ contains $F$.

Now suppose $F$ is an annulus, 4-holed sphere, or 1- or
2-holed torus. If $d_F(\Sigma) > 2g(\Lambda) +2$ then,
by Proposition \ref{neitherfacinglem}, $h$ cannot {\em weakly} $F$-split $f$.
By definition this means $h$ $F$-spans $f$, and again we are done by
Proposition \ref{spanning is good}. 

\end{proof}

\bibliographystyle{amsplain}
\bibliography{math}

\end{document}